\documentclass[reqno,a4paper]{amsart}

\usepackage{latexsym,xspace,enumerate,amsfonts,amsmath,amssymb,amscd}
\usepackage{xypic,color,tikz,hyperref,syntonly,MnSymbol,shadow,slashed,paralist,stmaryrd}
\usepackage[capitalize]{cleveref}
\usepackage[arrow, matrix, curve]{xy}

\newcommand{\Z}{\mathbb{Z}}

\newcommand{\R}{\mathbb{R}}

\newcommand{\GL}{\mr{GL}}

\newcommand{\SO}{\mr{SO}}

\newcommand{\SU}{\mr{SU}}
\newcommand{\Sp}{\mr{Sp}}
\newcommand{\Gt}{\mr{G_2}}
\newcommand{\Spin}{\mr{Spin}}

\newcommand{\Hol}{\mr{Hol}}

\newcommand{\mr}{\mathrm}

\newcommand{\mc}{\mathcal}

\renewcommand{\ker}{\mathop{\rm ker}\nolimits}

\newcommand{\Tr}{\mathop{\rm Tr}\nolimits}

\newcommand{\Diff}{\mbox{\sl Diff}}

\newcommand{\Ric}{\mathop{\rm Ric}\nolimits}

\newcommand{\tr}{\mbox{\rm tr}}

\newcommand{\note}[1]{\marginpar{\raggedright\if@twoside\ifodd\c@page\raggedleft\fi\fi\sf\scriptsize \red{RMK: #1}}}

\newcommand\red[1]{\textcolor{red}{#1}}

\newcommand{\be}{\begin{equation}}
\newcommand{\ben}{\begin{equation}\nonumber}
\newcommand{\ee}{\end{equation}}
\newcommand{\bp}{\begin{para}}
\newcommand{\ep}{\end{para}}
\newcommand{\bps}{\begin{paras}}
\newcommand{\eps}{\end{paras}}
\newcommand{\benum}{\begin{enumerate}[(i)]}
\newcommand{\eenum}{\end{enumerate}}

\def\Mpara{\mathcal{M}_{\para}}

\newtheorem{proposition}{\textbf{Proposition}}

\newtheorem*{lemma*}{\textbf{Lemma}}
\newtheorem{corollary}[proposition]{\textbf{Corollary}}
\newtheorem{theorem}[proposition]{\textbf{Theorem}}
\newtheorem*{theorem*}{\textbf{Theorem}}

\theoremstyle{definition}

\newtheorem*{example*}{\textbf{Example}}
\newtheorem*{remark}{\textbf{Remark}}

\newcounter{para}[section]
\newenvironment{para}[2][]{\refstepcounter{para}\noindent\ignorespaces{\bf #1\thepara. #2.} \rmfamily}{\noindent\ignorespacesafterend\bigskip}
\newenvironment{paras}[1]{\noindent\ignorespaces{\bf #1.} \rmfamily}{\noindent\ignorespacesafterend\bigskip}

\let\ol\overline
\let\witi\widetilde
\let\para\|

\begin{document}     
\title{Holonomy rigidity for Ricci-flat metrics}
 
\author{Bernd Ammann} 
\address{Fakult\"at f\"ur Mathematik \\
Universit\"at Regensburg \\ Universit\"atsstra{\ss}e 40 \\
D--93040 Regensburg \\  
Germany}
\email{bernd.ammann@mathematik.uni-regensburg.de}

\author{Klaus Kr\"oncke} 
\address{Fachbereich Mathematik der Universit\"at Hamburg \\
Bundesstr. 55 \\ D--20146 Hamburg \\ Germany}
\email{klaus.kroencke@uni-hamburg.de}

\author{Hartmut Wei\ss{}} 
\address{Mathematisches Seminar der Universit\"at Kiel\\ Ludewig-Meyn Stra{\ss}e 4\\ D--24098 Kiel\\ Germany}
\email{weiss@math.uni-kiel.de}

\author{Frederik Witt} 
\address{Institut f\"ur Geometrie und Topologie der Universit\"at Stuttgart\\ Pfaffenwaldring 57\\ D--70569 Stuttgart\\ Germany}
\email{frederik.witt@mathematik.uni-stuttgart.de}

\begin{abstract}
On a closed connected oriented manifold $M$ 
we study the space $\mathcal{M}_\|(M)$ of all Riemannian 
metrics which admit a non-zero parallel spinor on the universal covering. 
Such metrics are Ricci-flat, and all known Ricci-flat metrics are of this form.
We show the following: The space $\mathcal{M}_\|(M)$ is a smooth submanifold 
of the space of all metrics, 
and its premoduli space is a smooth finite-dimensional manifold.
The holonomy group is locally constant on $\mathcal{M}_\|(M)$. 
If $M$ is spin,
then the dimension of the space of parallel spinors is a locally constant function on~$\mathcal{M}_\|(M)$. 
\end{abstract}


\date{\today}


\maketitle
%

\section{Overview over the results}\label{intro}

Let $M$ be a compact connected oriented manifold without boundary, and let 
$\pi:\witi M\to M$ be its universal covering. We assume throughout the article
that $\witi M$ is spin. We define $\Mpara(M)$ to
be the space of all Riemannian metrics on $M$, such that $(\witi M,\tilde g)$, $\tilde g:=\pi^* g$, 
carries a (non-zero) parallel spinor. This implies that $g$ is a stable 
Ricci-flat metric on $M$, see \cite{dww05}. It is still an open question, 
whether  $\Mpara(M)$ contains all Ricci-flat metrics on~$M$.

In the past the space of such metrics was studied in much detail in the 
irreducible and simply-connected case, see e.g.\ \cite{wa89}, \cite{wa91}. Much less is known already in the non-simply-connected irreducible case, see e.g.\ \cite{wa95}. However, the general reducible case has virtually not been addressed at all. For example there is no classification of the full holonomy groups for metrics in $\Mpara(M)$ if the restricted holonomy is reducible.

The second named author recently found an efficient method to describe 
deformations of products of stable Ricci-flat manifolds \cite{kr15}.
In the present article we are using this method to see that the moduli space 
$\Mpara(M)$ is well-behaved. 

In fact we show that $\Mpara(M)$ is a smooth submanifold of the space of all metrics, and
its premoduli space is a smooth finite-dimensional manifold, 
see Corollary~\ref{cor.smooth}. If we view the dimension of the space of 
parallel spinors as a function $\Mpara(M)\to \mathbb{N}_0$, then this 
function is locally constant, see Corollary~\ref{dim.par.con}. Similarly, the full holonomy 
group, viewed as a function from $\Mpara(M)$ to conjugacy classes of subgroups
in $\GL(n,\R)$ is locally constant, see Theorem~\ref{theo.rig}.

\section{Ricci-flat metrics}\label{ric.fla.met}
{\bf Conventions.} All our manifolds will be connected and of dimension $n\geq3$ unless stated otherwise. A {\em spin} manifold is a manifold together with a fixed spin structure.

\bigskip

A Riemannian manifold $(M,g)$ is {\em Ricci-flat} if its Ricci tensor $\Ric^g$ vanishes identically. There are various reasons to study this class of metrics. 

It is already a challenging problem to decide whether a compact manifold admits
a Ricci-flat metric. Compact manifolds with positive Ricci 
curvature must have finite fundamental group due to the Bonnet-Myers theorem.
A similar obstruction is also available for compact manifolds with nonegative 
Ricci curvature: as a consequence of the splitting 
theorem by Cheeger and Gromoll \cite{chgr71}, one obtains the following theorem, see also 
Fischer and Wolf \cite{fiwo75}.

\begin{theorem*}[Structure theorem for Ricci-flat manifolds]
Let $(M,g)$ be a compact Ricci-flat manifold. Then there exists a finite normal Riemannian covering $\ol M \times T^q \to M$ with $(\ol M,\bar g)$ a compact simply-connected Ricci-flat manifold and $(T^q,g_{fl})$ a flat torus.
\end{theorem*}

The theorem implies that $\pi_1(M)$ contains a free abelian group $\Z^q$ of 
rank $q$ of finite index, and $q$ satisfies $b_1(M)\leq q\leq \dim M$, 
which acts by translations on $\R^q$ and trivially on $\ol M$.

A second reason for studying Ricci-flat metrics comes from the intimate relation with the concept of {\em holonomy}. To fix notation we briefly recall the definition of holonomy groups and some of their main properties. We refer to~\cite[Chapter 10]{be87} for further details. Let us fix a point $x\in M$ and an identification of $T_xM$ with $\R^n$. Up to conjugacy in $\GL(n,\R)$ these choices define a subgroup $\Hol(M,g)\subset\GL(n,\R)$ as the set of endomorphisms given by parallel transport around a loop in $x$, called the {\em (full) holonomy group} 
of $(M,g)$.
The subgroup obtained by taking loops homotopic to the constant path is called the {\em restricted holonomy group} and is denoted by $\Hol_0(M,g)$. Obviously, 
$\Hol_0(M,g)$ is a normal and connected subgroup of $\Hol(M,g)$, and we have 
$\Hol(\witi M,\tilde g)=\Hol_0(M,g)$. 
It is also known that $\Hol_0(M,g)$ is a 
closed Lie subgroup of $\SO(n)$, in contrast 
to $\Hol(M,g)$ which might be non-compact, see \cite{wilking:99}.
Parallel transport induces a group epimorphism $\pi_1(M)\to \Hol(M,g)/\Hol_0(M,g)$, and thus
$\Hol(M,g)/\Hol_0(M,g)$ is countable. As a consequence $\Hol_0(M,g)$ is 
the connected component of $\Hol(M,g)$ containing the identity, and 
the index of $\Hol_0(M,g)$ in $\Hol(M,g)$ is the number of connected components
of $\Hol(M,g)$.

For a locally irreducible, non-symmetric and Ricci-flat 
Riemannian manifold  $(M,g)$ of dimension~$n$, $\Hol_0(M,g)$ 
is conjugate to one of the groups occuring in Table~\ref{ric.fla.hol}. 
The possible full holonomy groups of such manifolds 
were classified by McInnes  \cite{mcinnes91}, see Wang \cite{wa95} 
for the spin case.
In the general Ricci-flat case, however, it is difficult to determine the 
full holonomy group. The full holonomy group of a non-irreducible Ricci-flat spin manifold does not have to be a product of holonomy groups in McInnes' list.
Still, it follows from the structure theorem for Ricci-flat manifolds,
that the subgroup $\Z^q\subset \pi_1(M)$ is in the kernel of the 
map $\pi_1(M)\to \Hol(M,g)/\Hol_0(M,g)$. Thus
$\Hol(M,g)/\Hol_0(M,g)$ is finite, see also~\cite[Theorem 6]{chgr71}.

A third reason comes from the connection with spin geometry (see for instance~\cite{fr00} for more background on spinors). A parallel spinor on any Riemannian spin manifold forces the underlying metric to be Ricci-flat with holonomy group strictly contained in $\SO(n)$. Conversely, as we have just recalled, a locally irreducible orientable Ricci-flat manifold is either of holonomy $\SO(n)$, or it has a simply-connected restricted holonomy group. In the latter case case the universal covering of $M$ is spin and there is a non-trivial parallel spinor on the universal covering. The dimension of the space of parallel spinors $\mc P(M,g)$ is determined by the holonomy group~\cite{wa89}, see Table~\ref{ric.fla.hol}. 
Furthermore, if the compact manifold $(M,g)$ carries 
a parallel spinor,
then it follows from Wang's work~\cite{wa91} that one obtains a parallel spinor
for any metric in the connected component of $g$ within 
the space of Ricci-flat metrics on~$M$.

In the whole article a spinor is a smooth section of the complex spinor bundle
and $\dim \mc P(M,g)$ is the complex dimension of the space of parallel spinors.

\begin{table}[b]
\begin{center}
\begin{tabular}{|l|c|}
\hline
$\Hol(M^n,g)$& $\dim\mc P(M,g)$\\ 
\hline
$\SO(n)$&--\\
$\SU(m),\,n=2m$&$2$\\
$\Sp(k),\,n=4k$&$k+1$\\
$\Spin(7),\,n=8$&$1$\\
$\Gt,\,n=7$&$1$\\
\hline
\end{tabular}
\end{center}
\caption{Special holonomy groups and the dimension of the space of 
parallel spinors.}
\label{ric.fla.hol}
\end{table}
%
Our main theorem studies the holonomy group of a metric with parallel spinor under Ricci-flat deformations.
\begin{theorem}[Rigidity of the holonomy group]\label{theo.rig}
Let $(M,g)$ be a compact Riemannian manifold whose universal covering
is spin and carries a parallel spinor. 
If~$g_t$, $t\in I:=[0,T]$ is a smooth family of Ricci-flat metrics such that $g_0=g$, then $\Hol(M,g_t)$ is conjugate to $\Hol(M,g)$ in $\GL(n,\R)$. 
\end{theorem}

\section{Rigidity of products}\label{sec.prod}

An important ingredient in the proof of Theorem~\ref{theo.rig}
is the following product formula by Kr\"oncke \cite{kr15}. It essentially
tells us that product metrics $g\times h$ on 
$N\times P$, with $g\in \Mpara(N)$, $h\in  \Mpara(P)$ and $N$, $P$ compact,  
are rigid in the sense that deformations within 
$\Mpara(N\times P)$ are again product metrics.
This cannot hold in the strict sense, as flat tori and 
pull-backs of $g\times h$ by non-product diffeomorphisms provide 
counterexamples to rigidity in the strict sense. However, rigidity 
holds modulo diffeomorphisms, provided that one of the factors 
does not carry parallel vector fields.

To describe this in detail, let $g$ be a Riemannian 
metric on a compact manifold $N$.
An infinitesimal deformation of $g$ given by a symmetric $2$-tensor $h$ is 
orthogonal to the conformal class if $h$ is trace-free
(i.e.\ $\tr_g h=0$), and $h$ is orthogonal 
to the diffeomorphism orbit of $g$ if $h$ is divergence-free.
On trace-free, divergence-free symmetric $2$-tensors the linearization of 
the Ricci curvature functional $g\mapsto \Ric^g$ 
is given by the Einstein operator $\Delta^N_E=\frac12 \nabla^*\nabla - \mathring{R} $, where 
$\mathring{R}h(X,Y):=\sum h(R_{e_i,X}Y,e_i)$ for a frame $(e_i)$ for $g$. 
For any smooth family $g_t$ of Ricci-flat metrics of some fixed volume such that $g_0=g$ and such that $h=\frac{d}{dt}\bigr|_{t=0}g_t$ is divergence-free, $h$ is in addition trace-free and $\Delta^N_Eh=0$, see \cite[Chapter 12]{be87}.

If a trace-free, divergence-free symmetric $2$-tensor $h$ is in $\ker(\Delta_{E}^{N})$, then it is called an 
\emph{infinitesimal Ricci-flat 
deformation}. An infinitesmal Ricci-flat deformation $h$ is called 
\emph{integrable} if there exists a smooth family $g_t$ of Ricci-flat metrics 
such that $g_0=g$ and
$\frac{d}{dt}\bigr|_{t=0}g_t=h$. We say that $g$ is a \emph{stable} Ricci-flat metric if $\Delta^N_E$ is 
positive-semidefinit on trace-free, divergence-free symmetric $2$-tensors. 
All metrics in $\Mpara(N)$ are stable, see \cite{wa91} and also \cite{dww05}.

\begin{theorem*}[Ricci-flat deformations of products, {\cite[Prop.~4.5 and 4.6]{kr15}}]
If $(N^n,g)$ and $(P^p,h)$ are two stable Ricci-flat manifolds, then
 $(N\times P,g+ h)$ is also stable.\\
Furthermore, on trace-free, divergence-free symmetric $2$-tensors we have
\begin{align*}
  \ker(\Delta_{E}^{N\times P})= \R(p\cdot g-n\cdot h)\oplus \left(\Gamma_\parallel(TN)\odot \Gamma_\parallel(TP)\right)                           \oplus\ker(\Delta_{E}^{N})\oplus \ker(\Delta_{E}^{P}),
\end{align*}
where  $\Gamma_\parallel$ is the space of parallel sections. Thus, if all infinitesimal Ricci-flat deformations of $(N,g)$ and $(P,h)$ are integrable, then all infinitesimal Ricci-flat deformations of $(N\times P,g+ h)$ are integrable.
\end{theorem*}

In particular, if $\Gamma_\parallel(TN)=0$ or if
$\Gamma_\parallel(TP)=0$, then Ricci-flat deformations of $g\times h$ are up to diffeomorphisms again of product form.

%
%
\section{Rigidity of the restricted holonomy group}
In order to prove the theorem we first prove the analogous statement for the restricted holonomy.

\begin{proposition}[Rigidity of the restricted holonomy group]\label{rig.res}
Let $(M,g)$ be a compact Riemannian manifold whose universal covering
is spin and carries a parallel spinor. 
If $g_t$, $t\in I:=[0,T]$ is a smooth family of Ricci-flat metrics such that $g_0=g$, then $\Hol_0(M,g_t)$ is conjugate to $\Hol_0(M,g)$ in $\GL(n,\R)$, i.e. there are $Q_t\in\GL(n,\R)$ with  $\Hol_0(M,g_t)=Q_t \Hol_0(M,g)Q_t^{-1}$. 
Moreover, the map $t\mapsto Q_t$ can be chosen continuously.
\end{proposition}
In the proposition we have fixed a base point 
$x\in M$ and an identification $T_xM\cong \R^n$. 
Thus $\Hol_0(M,g_t)\in \GL(n,\R)$.

\begin{remark}\

\begin{enumerate}[(i)]
\item The proposition (and our main theorem) apply for example to Riemannian 
spin manifolds carrying a parallel spinor. Under these assumptions, 
the universal covering 
$(\witi M,\tilde g)$ is spin as well, and the parallel spinor on $(M,g)$ 
lifts to a parallel 
spinor on $(\witi M,\tilde g)$.
\item The proposition (and our main theorem) are false for non-compact manifolds. Indeed, let $g$ be a (possibly non-complete) 
Ricci-flat metric on $\R^n$ with a parallel spinor and non-trivial holonomy group $H$. Let $\mu_s:\R^n\to \R^n$ be multiplication with $1-s$ for $1>s\geq0$. Then $g_s:=(1-s)^{-2}\mu_s^*g$ also has holonomy $H$ for $1>s\geq0$. However, for $s\to 1$ the metric $g_s$ converges in the compact-open $C^\infty$ topology to the flat metric $g_1$ with trivial holonomy.	
\end{enumerate}

\end{remark}

\begin{proof}[Proof of the proposition]
Consider the universal Riemannian covering $(\witi M,\tilde g)$ of $(M,g)$, and let $\Gamma\cong\pi_1(M)$ be the group of Deck transformations. 
Using the structure theorem for Ricci-flat metrics  
\cite{chgr71,fiwo75} recalled in \cref{ric.fla.met}, 
we know that $(M,g)$ has a finite normal Riemannian covering isometric to $(\ol M,\bar g) \times (T^q,g_{fl})$, where $(\ol M,\bar g)$ is compact, simply-connected and Ricci-flat, and $(T^q,g_{fl})$ a flat torus.
By the de Rham decomposition theorem, $(\ol M,\bar g)$ is globally isometric to a Riemannian product of compact, simply-connected and {\em irreducible} Ricci-flat manifolds $(\ol M_i,\bar g_i)$, $i=1,\ldots,r$. 
Thus we obtain a finite Riemannian covering
\[
(\ol M\times T^q=\ol M_1\times\ldots\times\ol M_r\times T^q,\bar g_1\times\ldots\times\bar g_r\times g_{fl})\to(M,g)
\]
whose holonomy group equals $H_1\times\ldots\times H_r=\Hol(\witi M,\tilde g)=\Hol_0(M,g)$.

Since $\witi M$ is spin, so is $\ol M$, and thus all $\ol M_i$ are spin. 
The spin structure on each $\ol M_i$ is unique as $\ol M_i$ is 
simply-connected.
On $T^q$ there are several distinct spin structures; we 
choose the unique spin structure that admits parallel spinors. 
A parallel spinor on a product corresponds to a product of parallel 
spinors (this follows from simple representation theoretic considerations 
as in \cite[Proposition 4.5]{ad96}, 
see for instance \cite[Theorem 2.5]{le00} for a detailed proof).
Thus the given parallel spinor on 
$\witi M$ implies the existence of parallel spinors on each  $\ol M_i$.
The holonomy group $H_i:=\Hol(\ol M_i,\bar g_i)$ is therefore conjugate 
to one of the groups of Table~\ref{ric.fla.hol}. 

Then by taking products, we obtain a parallel spinor on
$(\ol M\times T^q,\bar g\times g_{fl})$, which lifts to a parallel spinor on 
$\witi M$, and we can assume without loss of generality that this is the given 
parallel spinor on $\witi M$ discussed above.


\smallskip

Next consider the family of Ricci-flat metrics $g_t$ on $M$. The covering $\ol M\times T^q \to M$ can be chosen independently of~$t$.
And let $\bar g_t\times g_{fl,t}$ be the pullback of~$g_t$ to $\ol M\times T^q$ induced by $(M,g_t)$. 
The metric $\bar g=\bar g_1\times\ldots\times\bar g_r$ is a product of irreducible Ricci-flat metrics which admit a parallel spinor, and thus the metrics
$\bar g_i$ are stable. The factors $(\overline{M}_i,\bar g_i)$ 
are irreducible, hence defined by a torsion-free $G$-structure for one of the groups $G\neq\SO(n)$ taken from Table~\ref{ric.fla.hol}. It then follows from~\cite{no13} (building on work of~\cite{goto04}) that infinitesimal Ricci-flat 
deformations of~$\bar g_i$ are integrable. 
Finally, they admit no harmonic $1$-forms as simply-connectedness 
implies vanishing of the first Betti number $b_1$.
Therefore the theorem on Ricci-flat deformations of products in Section~\ref{sec.prod}
implies that the Ricci-flat deformation $\bar g_t$ of~$\bar g$ is --- up to pull-back by diffeomorphisms in $\Diff_0(\bar M)$ --- of the form $\bar g_{t,1}\times\ldots\times\bar g_{t,r}$, where $\bar g_{t,i}$ is a smooth Ricci-flat deformation of~$\bar g_i$. By the known rigidity results in the irreducible case~\cite{no13,wa91}, 
the holonomy group of nearby Ricci-flat deformations of~$\bar g_i$ will be conjugate to $H_i$, and the conjugating element can be chosen continuously in $t$.
\end{proof}
%
%
%

\section{Applications}

\begin{corollary}\label{dim.par.con}
Let $M$ be a compact spin manifold, and let $g_0 \in \Mpara(M)$. If $g_t$, $t\in I:=[0,T]$ is a smooth family of Ricci-flat metrics on $M$ such that $g_0=g$, then $g_t \in \Mpara(M)$ and $\dim\mc P(M,g_t)$ is constant in $t$.
\end{corollary}

\begin{remark}
A parallel spinor is a special case of a real Killing spinors. If $(M,g)$ 
carries a (non-zero) real Killing spinor, then $g$ is Einstein.
However, it was shown in \cite{coevering}, that our results do no longer hold
if we replace parallel spinors by real Killing spinors and Ricci-flat 
metrics by Einstein metrics.
\end{remark}

\begin{proof}
In the proof of \cref{rig.res} we have chosen a spin structure on $\ol M\times T^q$. As $(M,g_t)$ carries a parallel
spinor, this spin structure has to coincide with the pullback of the spin structure on $M$.

Let $\ol\Gamma$ be the (finite) group of Deck transformations of the covering space $\ol M \times T^q \to M$. By normality of this covering space, 
$\ol \Gamma$ acts transitively on each fiber, and this action lifts to the spin structure, and thus to the spinor bundle.
The dimension of $\mc P(\ol M\times T^q,\bar g_t\times g_{fl,t})$ is determined by the holonomy group 
$\Hol(\ol M,\bar g_t)=\Hol(\ol M\times T^q,\bar g_t\times g_{fl})=\Hol_0(M,g_t)$ which therefore does not depend on $t$. Further, $\mc P(M,g_t)$ 
can be identified with the subspace of $\mc P(\ol M\times T^q,\bar g_t\times g_{fl,t})$ which is invariant 
for the representation $\rho_t:\ol\Gamma\to \GL(\mc P(\ol M\times T^q,\bar g_t\times g_{fl,t}))$. 
Representation theory of finite groups implies 
\[
\dim\mc P(M,g_t)= \frac1{|\ol\Gamma|}\sum_{\gamma\in\ol\Gamma}\Tr\rho_t(\gamma)
\]
which is continuous in $t$, hence locally constant.
\end{proof}

This corollary extends Wang's result from~\cite{wa91} which states that the dimension is constant if $M$ is simply-connected and $g_0$ irreducible. An immediate consequence is that the premoduli space $\Mpara(M)/\Diff_0(M)$ is an open subset of the premoduli space of Ricci-flat metrics. In general, the premoduli space of Ricci-flat metrics 
is merely a real analytic subset of a finite-dimensional 
real analytic manifold, and the tangent 
space of this manifold can be identified with the space of infinitesimal 
Ricci-flat deformations and rescalings of the metric
(cf.\ for instance the discussion in~\cite[Chapter 12.F]{be87} 
building on Koiso's work~\cite{ko83}).

\begin{corollary}\label{cor.smooth}
Assume that $M$ is a compact spin manifold.
Then $\Mpara(M)/\Diff_0(M)$ is a smooth manifold.
Furthermore,  $\Mpara(M)$ is a smooth submanifold in the space of all metrics on $M$.
\end{corollary}

\begin{proof}
We have to check that every infinitesimal Ricci-flat deformation of $g_0 \in \Mpara(M)$ is in fact integrable. Consider again the finite cover $\ol M \times T^q \to M$. The space of Ricci-flat metrics on $\ol M \times T^q$ is smooth by the proof of Proposition \ref{rig.res}, and so is the space of Ricci-flat metrics on $M$ as the fixed point set of the finite group $\ol\Gamma$ acting on the former space by pull-back. This implies the result.
\end{proof}

This generalises Nordstr\"om's result \cite{no13} for metrics defined by a torsion-free $G$-structure for one of the groups $G\neq\SO(n)$ taken from Table~\ref{ric.fla.hol} to general metrics in $\Mpara(M)$. Nordstr\"om makes strong use of Goto's earlier result \cite{goto04} about the unobstructedness of the deformation theory of such $G$-structures.

\medskip

A further application involves the spinorial energy functional introduced in~\cite{aww15}, to which we refer for details. To define it, consider the space of sections $\mc N$ of the universal bundle of unit spinors over $M$. A section $\Phi\in\mc N$ can be thought of as a pair $(g,\phi)$ where $g$ is a Riemannian metric and $\phi\in\Gamma(\Sigma_gM)$ is a $g$-spinor of constant length one. The relevant functional is defined by
\ben
\mc{E}:\mc{N}\to\R_{\geq 0},\quad \Phi \mapsto\tfrac{1}{2}\int_M|\nabla^g\phi|_g^2\,dv^g,
\ee
where $\nabla^g$ denotes the Levi-Civita connection on the $g$-spinor bundle $\Sigma_gM$, $|\,\cdot\,|_g$ the pointwise norm on $T^*\!M \otimes \Sigma_gM$ and $dv^g$ the Riemann-Lebesgue measure given by the volume form of $g$. By~\cite[Corollary 4.10]{aww15} the set of critical points $\mr{Crit}(\mc E)$ consists precisely of $g$-parallel unit spinors $(g,\phi)$, provided that $\dim M>2$. Since by \cref{dim.par.con} the dimension of the space of parallel spinors is constant under Ricci-flat deformations, the proof of~\cite[Theorem 4.17]{aww15} immediately implies the subsequent

\begin{corollary}
The functional $\mc E$ is Morse-Bott, i.e.\ the critical set $\mr{Crit}(\mc E)$ is smooth and $\mc E$ is non-degenerate transverse to $\mr{Crit}(\mc E)$.
\end{corollary}


Now the universal covering group of $\Diff_0(M)$, the diffeomorphisms homotopic to the identity, naturally acts on $\mr{Crit}(\mc E)$. It also follows 
with the arguments above that the premoduli space of critical points, that is, $\mr{Crit}(\mc E)$ divided by this action, is smooth as well. 

\section{Rigidity of the full holonomy group}

We now prove Theorem~\ref{theo.rig}, using 
the following fact from Lie group theory.

\begin{theorem*}[Montgomery-Zippin~\cite{mozi42}]
Let $H_0$ be a compact subgroup of a Lie group~$G$. Then there exists an open neighbourhood $U$ of $H_0$ such that if $H$ is a compact subgroup of $G$ contained in $U$, then there exists $g\in G$ with $g^{-1}Hg\subset H_0$. Moreover, upon sufficiently shrinking $U$, $g$ can be chosen in any neighbourhood of the identity of $G$.
\end{theorem*}

\begin{proof}[Proof of Theorem~\ref{theo.rig}]
The group epimorphisms 
$$\alpha_t:\Gamma=\pi_1(M)\to\Hol(M,g_t)/\Hol_0(M,g_t)$$ 
factor through 
$\Gamma\to \ol\Gamma=\Gamma/\Z^q$ to epimorphisms 
$\bar\alpha_t:\ol\Gamma\to\Hol(M,g_t)/\Hol_0(M,g_t)$. 
Choose loops $\gamma_1,\ldots,\gamma_r:[0,1]\to M$ where $\gamma_i(0)=\gamma_i(1)$ is the base point of $M$ and such that 
$\ol\Gamma=\{[\gamma_1],\ldots,[\gamma_\ell]\}$, $\ell=\#\ol\Gamma$. Let $A_{t,i}$ be the parallel 
transport along $\gamma_i$ for the metric $g_t$.

Thus
\[
\Hol(M,g_t)= A_{t,1}\cdot\Hol_0(M,g_t)\cup\ldots \cup A_{t,\ell}\cdot\Hol_0(M,g_t). 
\]
Let $k(t):=\#\{i\in\{1,2,\ldots,\ell\}\mid A_{t,i}\in\Hol_0(M,g_t)\}=\#\ker \ol\alpha_t$. 
Then the index of $\Hol_0(M,g_t)$ in $\Hol(M,g_t)$ is $\ell/k(t)$ and this index 
is the number of connected components of $\Hol(M,g_t)$.
Thus the index of $\Hol_0(M,g_t)$ in $\Hol(M,g_t)$ is 
lower semi-continuous. We fix $t_0\in [0,T]$.
As $A_{t,i}$ depends continuously on $t$, 
the theorem of Montgomery-Zippin
implies that $\Hol(M,g_{t})$ is conjugate
to a subgroup $H_t$ of $\Hol(M,g_{t_0})$ if $t$ is sufficiently close to $t_0$.
Conjugation preserves the connected component containing the identity,   
thus $\Hol_0(M,g_{t_0})\subset H_t\subset \Hol(M,g_{t_0})$. 
Lower semi-continuity of the number of connected components of $H_t$ 
implies the result.
\end{proof}


\end{document}